\documentclass[11pt]{amsart}

\usepackage[a4paper,hmargin=3.5cm,vmargin=4cm]{geometry}
\usepackage{amsfonts,amssymb,amscd,amstext}
\usepackage{graphicx}
\usepackage{color}
\usepackage[dvips]{epsfig}

\usepackage[utf8]{inputenc}
\usepackage{hyperref}

\usepackage{verbatim}
\usepackage{fancyhdr}
\input xy
\xyoption{all}

\usepackage{times}
\usepackage{enumerate}
\usepackage{titlesec}
\usepackage{mathrsfs}

\numberwithin{equation}{section}
\numberwithin{figure}{section}

\titleformat{\section}
{\filcenter\bfseries\large} {\thesection{.}}{0.2cm}{}
\titleformat{\subsection}[runin]
{\bfseries} {\thesubsection{.}}{0.15cm}{}[.]
\titleformat{\subsubsection}[runin]
{\em}{\thesubsubsection{.}}{0.15cm}{}[.]

\usepackage[up,bf]{caption}

\newtheorem{theorem}{Theorem}[section]

\newtheorem{lemma}[theorem]{Lemma}

\theoremstyle{definition}

\newtheorem{remark}[theorem]{Remark}

\newcommand\B{\mathbb{B}}
\newcommand\C{\mathbb{C}}

\newcommand\N{\mathbb{N}}

\newcommand\igot{\mathfrak{i}}

\renewcommand\igot{\mathfrak{i}}

\renewcommand\imath{\igot}

\newcommand\wt{\widetilde}

\newcommand\di{\partial}

\newcommand\dist{\mathrm{dist}}

\newcommand\Aut{\mathrm{Aut}}

\def\dist{\mathrm{dist}}

\def\Ell1{\mathrm{Ell_1}}
\def\CEll1{\mathrm{CEll_1}}

\begin{document}

\fancyhead[LO]{Holomorphic families of Fatou-Bieberbach domains and applications to Oka manifolds}
\fancyhead[RE]{F.\ Forstneri\v c, E. F. \ Wold} 
\fancyhead[RO,LE]{\thepage}

\thispagestyle{empty}


\begin{center}
{\bf\LARGE Holomorphic families of Fatou-Bieberbach domains and applications to Oka manifolds}

\vspace*{0.5cm}

{\large\bf  Franc Forstneri{\v c} and Erlend Forn\ae ss Wold} 
\end{center}

\vspace*{0.5cm}

{\small
\noindent {\bf Abstract}\hspace*{0.1cm}
We construct holomorphically varying families of Fatou-Bieberbach domains
with given centres in the complement of any compact polynomially convex subset $K$ of $\C^n$
for $n>1$. This provides a simple proof of the recent result of Y.\ Kusakabe to the effect 
that the complement $\C^n\setminus K$ of any polynomially convex subset $K$ of $\C^n$
is an Oka manifold. The analogous result is obtained with $\C^n$ replaced by 
any Stein manifold with the density property.
}

\noindent{\bf Keywords}\hspace*{0.1cm} Fatou-Bieberbach domain, polynomially convex set, Oka manifold 
\vspace*{0.1cm}

\noindent{\bf MSC (2010):}\hspace*{0.1cm}  32H02, 32M17, 32Q56
%
%

\noindent {\bf Date: \rm 4 Februar 2021}


%
%
\section{Introduction}\label{sec:intro}
A {\em Fatou-Bieberbach domain} in $\C^n$ is a proper subdomain $\Omega\subsetneq\C^n$
which is biholomorphic to $\C^n$. No such domains exists for $n=1$, but they are plentiful
for any $n>1$; see the survey of this topic in \cite[Chapter 4]{Forstneric2017E}.
In particular, the basin of attraction of an attracting fixed point of
a holomorphic automorphism of $\C^n$ (or in fact of any complex manifold) is 
biholomorphic to $\C^n$, cf.\ \cite{RosayRudin1988} and \cite[Theorem 4.3.2]{Forstneric2017E}. 
Furthermore, for any compact polynomially convex set $K\subset \C^n$
$(n>1)$ and point $p\in \C^n\setminus K$ there is a Fatou-Bieberbach domain 
$\Omega\subset \C^n$ such that $p\in \Omega$ and $K\cap \Omega=\varnothing$;
this is a special case of \cite[Proposition 9]{ForstnericRitter2014} where the same is shown
with $p$ replaced by any compact convex set.

In this note we prove the following more general result in this direction.

%
%
\begin{theorem}\label{th:main}
Let $K$ be a compact polynomially convex set in $\C^n$ for some $n>1$, 
$L$ be a compact polynomially convex set in $\C^N$ for some $N\in\N$, 
and $f:U\to \C^n$ be a holomorphic map
on an open neighbourhood $U\subset \C^N$ of $L$ such that $f(z)\in \C^n\setminus K$
for all $z\in L$. Then there are an open neighbourhood $V\subset U$ of $L$
and a holomorphic map $F:V\times \C^n\to \C^n$ such that for every $z\in V$
we have that $F(z,0)=f(z)$ and the map $F(z,\cdotp):\C^n\to \C^n\setminus K$ is injective.
Hence, $\Omega_z:=\{F(z,\zeta):\zeta\in\C^n\}$ is a Fatou-Bieberbach domain 
in $\C^n\setminus K$ for each $z\in V$.
\end{theorem}

A proof of this result based solely on Anders\'en-Lempert theory is given in Section \ref{sec:proof};
it also applies if $\C^N$ is replaced by an arbitrary Stein manifold, and also to variable
fibres $K_z\subset \C^n$, $z\in L$, with polynomially convex graph (see Remark \ref{rem:variable}). 
For a convex parameter space $L\subset \C^N$ we prove the analogous result with $\C^n$ 
replaced by an arbitrary Stein manifold having the density property; see Theorem \ref{th:main2}. 

These two theorems immediately imply 
the following recent and very interesting result of Yuta Kusakabe.

%
%
\begin{theorem}{\rm (Kusakabe, \cite[Theorem 1.2 and Corollary 1.3]{Kusakabe2020complements}.)} \label{th:Oka}
For any compact holomorphically convex subset $K$ in a Stein manifold $Y$ with the density property
the complement $Y\setminus K$ is an Oka manifold. In particular, the complement $\C^n\setminus K$
of any compact polynomially convex set $K$ in $\C^n$ for $n>1$ is an Oka manifold. 
\end{theorem}

This is the first result in the literature which gives a large class of 
Oka domains in $\C^n$ for any $n>1$, and it provides an affirmative answer to a long-standing problem. 
As  noted in \cite[Corollary 1.4]{Kusakabe2020complements},
it follows from Theorem \ref{th:Oka} and \cite[Theorem 1.1]{Forstneric2017Indam} 
that for any compact polynomially convex set $K$ in $\C^n$ $(n>1)$, the complement $\C^n\setminus K$ 
(like any $n$-dimensional Oka manifold) is the image of a strongly dominating 
holomorphic map $\C^n\to\C^n\setminus K$.

Recall that a complex manifold $Y$ is said to be an {\em Oka manifold} if every holomorphic map 
from a neigbourhood of a compact (geometrically) convex set $L$ in a Euclidean space $\C^N$ 
into $Y$ is a uniform limit on $L$ of entire maps $\C^N \to Y$ 
(see \cite[Definition 5.4.1]{Forstneric2017E}; this is also called the 
{\em convex approximation property} and denoted CAP). By \cite[Theorem 5.4.4]{Forstneric2017E},
holomorphic maps $S\to Y$ from any reduced Stein space $S$ to an Oka manifold $Y$
satisfy all natural Oka-type properties. In his recent paper \cite{Kusakabe2018Oka},
Kusakabe showed that a complex manifold $Y$ is Oka if (and only if) is satisfies
the following condition:

($*$) For any compact convex set $L\subset \C^N$, open set $U\subset \C^N$ containing $L$,  
and holomorphic map $f:U\to Y$ there are an open set $V$ with $L\subset V\subset U$ and 
a holomorphic map $F:V\times \C^N\to Y$ with $F(\cdotp,0)=f|_V$ such that
\[
	\frac{\di}{\di \zeta}\Big|_{\zeta=0} F(z,\zeta):\C^N\to T_{f(z)}Y
	\ \ \text{is surjective for every}\ \ z\in V.
\]
A map $F$ with these properties is called a {\em dominating holomorphic spray over $f|_V$}.
This is a restricted version of condition Ell$_1$ introduced by Gromov \cite[p.\ 72]{Gromov1986}
(see also \cite{Gromov1989}). In \cite{Kusakabe2018Oka}, Kusakabe used the technique of
gluing sprays from \cite[Sect.\ 5.9]{Forstneric2017E} to show that this condition implies CAP,
so $Y$ is an Oka manifold. Conversely, it has been known before that 
every Oka manifold satisfies condition $\Ell1$ (see \cite[Corollary 8.8.7]{Forstneric2017E}).

Theorem \ref{th:main} provides a very special dominating spray with values in 
$\C^n\setminus K$ over any given holomorphic map $f:L\to \C^n\setminus K$, thereby proving 
Theorem \ref{th:Oka} in the case $Y=\C^n$. In exactly the same way, Theorem \ref{th:main2} implies the
general case of Theorem \ref{th:Oka}. 

Kusakabe also proved in \cite[Theorem 4.2]{Kusakabe2020complements} that certain closed 
noncompact sets in Stein manifolds $Y$ with the density property have Oka complements.
He constructed a holomorphically varying family $f(z)\in \Omega_z \subset Y\setminus K$ 
$(z\in L)$ of nonautonomous basins with uniform bounds 
(i.e., basins of random sequences of automorphisms of $Y$ which are uniformly attracting 
at $f(z)\in Y \setminus K$); these are elliptic manifolds as shown by 
Forn\ae ss and Wold \cite{FornaessWold2016}, hence Oka. When $Y=\C^n$,
the domains $\Omega_z$ can be chosen Fatou-Bieberbach domains by using 
Theorem \ref{th:main} with variable fibres (cf.\ Remark \ref{rem:variable}).
Kusakabe's proof of \cite[Theorem 4.2]{Kusakabe2020complements}
can also be modified so as to provide a family of Fatou-Bieberbach domains in the general
situation under consideration.

%
%
\section{Proof of Theorem \ref{th:main}}\label{sec:proof}
We shall use some standard facts concerning polynomial convex sets; we
refer the reader to the monograph by E.\ L.\ Stout \cite{Stout2007}. 
Firstly, if $K_1,K_2\subset \C^n$ is a pair of disjoint compact sets such that $K_1\cup K_2$ is 
polynomially convex, then for any polynomially convex set $K'_1\subset K_1$ the union 
$K'_1\cup K_2$ is also polynomially convex. Secondly, 
every compact polynomially convex set $K\subset \C^n$ is 
the zero set of a nonnegative plurisubharmonic exhaustion function $\rho:\C^n\to[0,+\infty)$ 
which is strongly plurisubharmonic on $\C^n\setminus K$. Choosing a sequence
$c_1>c_2>\cdots>0$ with $\lim_{i\to\infty} c_i=0$ and setting
$K_i=\{\rho\le c_i\}\supset K$ yields a decreasing sequence of compact polynomially convex
sets with $K_{i+1}$ contained in the interior of $K_i$ for every $i\in\N$.

Let $f$, $K$ and $L$ be as in the theorem. We replace $L$ by a slightly bigger polynomially convex set 
(still denoted $L$) contained in $U$ and such that $f(z)\in \C^n\setminus K$ for all $z\in L$. 
Choose a sequence $K_i\supset K$ as above, with $K_1$ chosen close enough to $K$ such that 
$f(z)\in  \C^n\setminus K_1$ for every $z\in L$.
The compact set $L\times K_i \subset\C^{N+n}$ is polynomially convex for every $i\in\N$.
Applying the change of coordinates $\psi(z,\zeta)=(z,\zeta-f(z))$ replaces $f$ by the zero function, 
and for every $i\in\N$ the set 
\begin{equation}\label{eq:Si}
	S_i = \psi(L\times K_i) \subset L\times \C^n \subset \C^{N+n}
\end{equation}
is polynomially convex and does not intersect $\C^N\times \{0\}^n$.  
Hence, $(L\times \{0\}^n)\cup S_1$ is polynomially convex.
Therefore, there is a small closed ball $B\subset \C^n$ centred at $0\in\C^n$ 
such that $(L\times B)\cap S_1=\varnothing$ and $(L\times B)\cup S_1$ is polynomially convex.
Since $S_i\subset S_1$ is polynomially convex, it follows that 
$(L\times B)\cup S_i$ is polynomially convex for each $i\in\N$. 

The following lemma will be used in the inductive construction.

\begin{lemma}\label{lem:main}
(Assumptions as above.) 
Let $B'\subset \C^n$ be a closed ball centred at the origin with $B\subset (B')^\circ$.
Then, there are an open neighbourhood $U'\subset U$ of $L$ and a biholomorphic 
map $\Phi:U'\times \C^n\to U'\times \C^n$ of the form $\Phi(z,\zeta)=(z,\phi(z,\zeta))$ such that 
\begin{enumerate}[\rm (a)]
\item $\Phi$ approximates the identity map as closely as desired on $L\times B$ and 
$\Phi(z,0)=(z,0)$ for all $z\in U'$,
\item $\Phi(S_1) \cap (L\times B')=\varnothing$, and 
\item the set $\Phi(S_2) \cup (L\times B')$ is polynomially convex.
\end{enumerate}
\end{lemma}

\begin{proof}
Choose $r>1$ such that $rB=B'$. Let $\theta_t(z,\zeta)=(z,t\zeta)$ for $z\in\C^N$,  
$\zeta\in\C^n$, and $t\in\C$.  We have that $(L\times B') \cap \theta_r(S_1) =\varnothing$ and 
\begin{equation}\label{eq:pc}
	(L\times B') \cup \theta_r(S_1) = \theta_r((L\times B)\cup S_1)
	\ \ \text{is polynomially convex}.
\end{equation}
Consider the isotopy of biholomorphic maps $\phi_t$ on a neighbourhood of 
$(L\times B)\cup S_1$ in $\C^{N+n}$ for $t\in [1,r]$ which equals the identity map on a neighbourhood 
of $L\times B$ and equals $\theta_t$ on a neighbourhood of $S_1$. Note that 
$\theta_t(S_1)$ is disjoint from $L\times B$ and the union 
$(L\times B) \cup \theta_t(S_1)$ is polynomially convex for all $t\in [1,r]$ 
(since it is contained in $\theta_t((L\times B) \cup S_1)=(L\times tB) \cup \theta_t(S_1)$ 
which is polynomially convex). Hence, by the parametric Anders\'en-Lempert theorem 
(see Kutzschebauch \cite{Kutzschebauch2005} or \cite[Theorem 4.12.3]{Forstneric2017E})
there is a holomorphic automorphism of $\C^{N+n}$ of the form $\Phi(z,\zeta)=(z,\phi(z,\zeta))$  
which approximates the identity map on $L\times B$, it agrees with the identity on $L\times \{0\}^n$, 
and it approximates $\theta_r$ on $S_1$. Hence, conditions (a) and
(b) in the lemma hold. Assuming that the approximations are close enough,
we have $\Phi(S_2)\subset \theta_r(S_1)$. Note that $\Phi(S_2)$ is polynomially
convex. In view of \eqref{eq:pc} it follows that $\Phi(S_2) \cup (L\times B')$ is polynomially convex
as well which gives condition (c).
\end{proof}

\begin{proof}[Proof of Theorem \ref{th:main}]
We apply the push-out method described in \cite[Section 4.4]{Forstneric2017E}.
Using Lemma \ref{lem:main} we inductively construct a decreasing sequence of open neighborhoods 
$U_k$ of $L$ and holomorphic automorphisms $\Phi_k(z,\zeta)=(z,\phi_k(z,\zeta))$ 
of $U_k\times \C^n$ such that, setting 
\[
	\Phi^k=\Phi_k\circ\Phi_{k-1}\circ \cdots \circ\Phi_1:U_k\times \C^n\to U_k\times \C^n,
\] 
the following conditions hold for every $k\in \N$.
\begin{enumerate}[\rm (i)]
\item $\Phi_k$ approximates the identity map as closely as desired on $L\times kB$ and
$\Phi_k(z,0)=(z,0)$ for all $z\in U_k$.
\item $\Phi^k(S_{k}) \cap (L\times (k+1)B)=\varnothing$. (Here, $S_k$ is given by \eqref{eq:Si}.)
\item The set $\Phi^k(S_{k+1}) \cup (L\times (k+1)B)$ is polynomially convex.
\end{enumerate}
Indeed, Lemma \ref{lem:main} furnishes the first map $\Phi_1$ with $B'=2B$ and the sets $S_2\subset S_1$; 
every subsequent step is of the same form by just increasing the indices.
Assuming that the approximations are close enough, 
\cite[Proposition 4.4.1 and Corollary 4.4.2]{Forstneric2017E} show that 
the limit $\Phi=\lim_{k\to\infty }\Phi^k$ exists uniformly on compacts on the domain 
\[
	\Omega=\bigl\{(z,\zeta)\in L\times \C^n: \Phi^k(z,\zeta)\ \text{is a bounded sequence}\bigr\}
	= \bigcup_{k=1}^\infty (\Phi^k)^{-1}(L\times kB), 
\]
and for every $z\in L$, $\Phi(z,\cdotp)$ maps the fibre $\Omega_z=\{\zeta\in \C^n:(z,\zeta)\in\Omega\}$
biholomorphically onto $\C^n$. By condition (ii) the set $S=\psi(L\times K)$ does not intersect 
$\Omega$ (it has been pushed to infinity by the sequence $\Phi^k$). 
Hence, the inverse map $\Phi^{-1}(z,\zeta) =(z,\varphi(z,\zeta))$ provides
a holomorphic family of Fatou-Bieberbach maps $\varphi(z,\cdotp):\C^n\to \C^n$
$(z\in L)$ such that $\varphi(z,0)=0$ and its image does not intersect the set 
$K-f(z)$. The function $F(z,\zeta)=\varphi(z,\zeta)+f(z)$ for $z\in L$ and $\zeta\in\C^n$
satisfies the conclusion of the theorem.
\end{proof}

\begin{remark}\label{rem:variable}
The above proof also applies in the case when the product $L\times K$ is replaced
by a compact polynomially convex set $\wt K\subset \C^{N+n}$ projecting onto $L$ 
whose fibres $K_z$ $(z\in L)$ depend on $z$. The conclusion remains the same,
that is, given a holomorphic map $f:L\to\C^n$ with $f(z)\in\C^n\setminus K_z$
for all $z\in L$, there is a holomorphically variable family of Fatou-Bieberbach
domains $f(z)\in \Omega_z\subset \C^n\setminus K_z$ for all $z\in L$.
\end{remark}

%
%
\section{Fatou-Bieberbach domains in Stein manifolds with the density property}

In this section we give a version of Theorem \ref{th:main} with $\mathbb C^n$ replaced
by an arbitrary Stein manifold with the density property. (See Varolin \cite{Varolin2001} 
or \cite[Definition 4.10.1]{Forstneric2017E} for this notion.) Every such manifold has dimension 
$>1$. The following result is similar to Theorem \ref{th:main}, but we impose 
the extra condition that the set $L$ is geometrically convex.  

\begin{theorem}\label{th:main2}
Let $X$ be a Stein manifold with the density property, $K$ be a compact holomorphically convex set 
in $X$, $L$ be a compact convex set in $\C^N$ for some $N\in\N$, and $f:U\to X$ be a  
holomorphic map on an open neighbourhood $U\subset \C^N$ of $L$ such that $f(z)\in X\setminus K$
for all $z\in L$. Then there are a neighbourhood $V\subset U$ of $L$
and a holomorphic map $F:V\times \C^n\to X$ with $n=\dim X$ such that for every $z\in V$
we have that $F(z,0)=f(z)$ and the map $F(z,\cdotp):\C^n\to X\setminus K$ is injective.
\end{theorem}

Hence, $\Omega_z:=\{F(z,\zeta):\zeta\in\C^n\}\subset X\setminus K$ is a Fatou-Bieberbach domain 
of the first kind (i.e., biholomorphic to $\C^n$) for each $z\in V$.

The proof of Theorem \ref{th:main2} depends on the following interpolation result for graphs. 
We denote by $\dist_X$ a distance function on $X$ compatible with the manifold topology.

\begin{lemma}\label{fibaut}
Let $X,K,L,U$ and $f$ be as above, and let $z_0\in L$ be arbitrary. 
Then for any $\epsilon>0$ there exist a neighbourhood $V\subset U$ of $L$ and 
a fibred holomorphic automorphism $\phi(z,x)=(z,\varphi(z,x))$ of $V\times X$ such that 
$\phi(z,f(z)) = (z,f(z_0))$ for all $z\in V$, and $\dist_X (\varphi(z,x),x) <\epsilon$ for all $z\in L$ and $x\in K$. 
\end{lemma}

\begin{proof}
We may assume that $z_0=0\in\C^N$. 
Let $V_1,\dots,V_m$ be complete holomorphic vector fields on $X$ such that $V_1(x),\ldots,V_m(x)$
span the tangent space $T_xX$ for all $x\in X$ (such exist  by \cite[Proposition 5.6.23]{Forstneric2017E}
since $X$ is Stein and has the density property). Let $\psi_{1,s},\dots,\psi_{m,s}$ denote their 
respective flows, $s\in \mathbb C$. Consider the map $\Psi:\C^m\times X\rightarrow X$ defined 
for $s=(s_1,\ldots,s_m)\in\C^m$ and $x\in X$ by 
\[
	\Psi(s_1,\dots,s_m,x) = \psi_{m,s_m}\circ\cdots\circ\psi_{1,s_1}(x). 
\]
Note that $\Psi(s,\cdotp)\in\Aut(X)$ for every $s\in\C^m$. 
Then the partial differential $\di_s|_{s=0}\Psi(s,f(0))$ has maximal rank $n=\dim X$, so there exists 
an $n$-dimensional linear subspace $\Lambda\subset\mathbb C^m$ on which this
differential has rank $n$. We may assume that $\Lambda=\C^n\times \{0\}^{m-n}$.
Write $s=(s',s'')$, with $s'\in\C^n$ and $s''\in \C^{m-n}$, and set $\wt \Psi_{s'}:= \Psi((s',0''),\cdotp) \in \Aut(X)$. 
It follows that there exists $\delta>0$ such that the map $s'\mapsto \wt \Psi_{s'}(f(0))$ 
is an embedding of the open $\delta$-ball $B_\delta \subset \C^n$ centred at $0\in\C^n$ 
onto an open neighbourhood of $f(0)\in X$.  

We replace $L$ by a slightly larger convex set $L'\subset U$ with $L\subset (L')^\circ$ without
changing the notation. We also choose a compact holomorphically convex  set $K'\subset X$
containing $K$ in its interior and such that $f(z)\in X\setminus K'$ for all $z\in L$.
Set $f_t(z)=f(t\cdot z)$ for $z\in L$ and $t\in [0,1]$. Consider the isotopy 
$\phi_t(z,x)$ defined to be the identity near $L\times K'$ and $\phi_t(z,f(z))=(z,f_{1-t}(z))$, $0\leq t\leq 1$
on the graph $Z:=\{(z,f(z)):z\in L\}\subset \C^N\times X$. The image of $\phi_t$ 
is the disjoint union of $L\times K'$ and the holomorphic graph 
of $f_{1-t}$ over $L$, so it is holomorphically convex in $\C^N\times X$. 
By using  \cite[Proposition 3.3.2]{Forstneric2017E}
(a fibred version of the tubular neighbourhood theorem for Stein manifolds)
along with the Oka-Grauert principle we can extend $\phi_t$ to a fibred isotopy of injective 
holomorphic maps on an open neighbourhood of $Z$ in $\C^N\times X$. 
Since $X$ has the density property, given $\eta>0$ there is 
a fibred holomorphic automorphism $\tilde\phi(z,x)=(z,\tilde\varphi(z,x))$ of $L\times X$
such that $\mathrm{dist}_X(\tilde\varphi(z,f_1(z)),f(0))<\eta$ for $z\in L$ and  
$\dist_X(\tilde\varphi(z,x),x)<\epsilon/2$ for $z\in L$ and $x\in K'$ (see \cite{Kutzschebauch2005} 
and \cite[Theorems 4.10.5 and 4.12.3]{Forstneric2017E}). 
Note that $f_1=f$. If $\eta>0$ is chosen small enough, 
there exists for each $z\in L$ a unique point $\lambda(z)\in B_\delta\subset\C^n$ such that 
$\wt \Psi_{\lambda(z)}(f(0))=\tilde\varphi(z,f(z))$. The fibred holomorphic automorphism
\[
	\phi(z,x) = \left(z,\wt\Psi_{\lambda(z)}^{-1}(\tilde\varphi(z,x))\right),\quad z\in L,\ x\in X
\]
then satisfies the lemma provided $\eta>0$ is chosen small enough.
\end{proof}

We will also need the following basic result which we include lacking a reference. (The existence 
of a Fatou-Bierberbach domain of the first kind containing a point $p\in X$ was proved by Varolin
\cite{Varolin2000}, but this is not sufficient for our purpose.) 

\begin{lemma}\label{fb}
Let $X$ be a Stein manifold with the density property, let $K\subset X$ be a holomorphically convex 
compact set, and let $p\in X\setminus K$. Then there exists a Fatou-Bieberbach domain 
$\Omega\subset X\setminus K$ of the first kind such that $p\in \Omega$. 
\end{lemma}

\begin{proof}
Let $K'$ be a holomorphically convex compact set in $X$ containing $K$ in its interior 
and such that $p\notin K'$.  Choose a local holomorphic coordinate $\phi:U_p\rightarrow\mathbb C^n$ on $X$
such that $p\in U_p\subset X\setminus K'$ and $\phi(p)=0$. 
Denote by $\B_\delta(0)$ the ball of radius $\delta$ centred at the origin in $\C^n$.
Let $\delta>0$ be chosen small enough such that $\overline{\B_\delta(0)} \subset  \phi(U_p)$ and
$\phi^{-1}(\overline{\B_\delta(0)})\cup K'$ is holomorphically convex in $X$. 
Let $F:\C^n\to\C^n$ be the automorphism $F(z_1,...z_n)=(\frac{z_1}{2},...,\frac{z_n}{2})$. 
Since $X$ has the density property, we can apply \cite[Theorem 4.10.5]{Forstneric2017E} to approximate 
the map which equals $\phi^{-1}\circ F\circ\phi$ on a neighbourhood of $\phi^{-1}(\overline{\B_\delta(0)})$ 
and equals the identity map on a neighbourhood of $K'$ 
to obtain a sequence $G_j\in\mathrm{Aut}(X)$ such that for any $k\in\mathbb N$ we have that 
$G_k\circ\cdots \circ G_1(K)\subset K'$, and setting $F_j=\phi\circ G_j\circ\phi^{-1}$ we have that 
\[
	s\cdot\|z\|\leq \|F_j(z)\|\leq r\cdot\|z\|,\quad j\in\N
\]
on $\B_\delta(0)$, with $0<s<\frac12 < r<1$ and $r^2<s<1$.  Now, following \cite[proof of Theorem 4]{Wold2005}
we have that the abstract basin of attraction, or the tail space $\wt \Omega$ 
(see \cite{FornaessStensones2004}) associated to $\{F_j\}_{j\in\N}$, is biholomorphic to $\mathbb C^n$, 
and the basin of attraction $\Omega$ of the sequence $\{G_j\}_{j\in\N}$ is biholomorphic to $\wt \Omega$. 
\end{proof}

\emph{Proof of Theorem \ref{th:main2}:} This is an immediate 
consequence of Lemmas \ref{fibaut} and \ref{fb}. 
$\hfill\square$

%
%
%
%
\subsection*{Acknowledgements}
The first named author is supported by the research program P1-0291 and grant J1-9104
from ARRS, Republic of Slovenia, and by the Stefan Bergman Prize 2019 from the 
American Mathematical Society. The second named author is supported by the 
RCN grant 240569 from Norway.




\newpage
\noindent Franc Forstneri\v c

\noindent Faculty of Mathematics and Physics, University of Ljubljana, Jadranska 19, SI--1000 Ljubljana, Slovenia

\noindent 
Institute of Mathematics, Physics and Mechanics, Jadranska 19, SI--1000 Ljubljana, Slovenia.

\noindent e-mail: {\tt franc.forstneric@fmf.uni-lj.si}

\vspace*{5mm}
\noindent Erlend Forn\ae ss Wold

\noindent Department of Mathematics, University of Oslo, PO-BOX 1053, Blindern, 0316 Oslo. Norway. 

\noindent e-mail: {\tt erlendfw@math.uio.no}

\end{document}